\documentclass[12pt,reqno]{amsart}

\usepackage{epsfig}
\usepackage{amsmath, amsthm, amsfonts, amssymb}
\usepackage{graphicx}

\numberwithin{equation}{section}

\newcommand{\C}{{\mathbb C}}
\newcommand{\N}{{\mathbb N}}

\newcommand{\Z}{{\mathbb Z}}

\newcommand{\K}{{\mathbb K}}

\newtheorem{theo}{{\sc \bf Theorem}}[section]

\newtheorem{lem}[theo]{{\sc \bf Lemma}}
\newtheorem{prop}[theo]{{\sc \bf Proposition}}

\newenvironment{defin}{\medskip\noindent{\bf Definition:\/} }{\medskip}

\begin{document}

\title{Derivations and Spectral Triples on Quantum Domains II: Quantum Annulus}

\author[Klimek]{Slawomir Klimek}
\address{Department of Mathematical Sciences,
Indiana University-Purdue University Indianapolis,
402 N. Blackford St., Indianapolis, IN 46202, U.S.A.}
\email{sklimek@math.iupui.edu}

\author[McBride]{Matt McBride}
\address{Department of Mathematics and Statistics,
Mississippi State University,
175 President's Cir., Mississippi State, MS 39762, U.S.A.}
\email{mmcbride@math.msstate.edu}

\author[Rathnayake]{Sumedha Rathnayake}
\address{Department of Mathematics,
University of Michigan,
530 Church St., Ann Arbor, MI 48109, U.S.A.}
\email{sumedhar@umich.edu}

\thanks{}

\date{\today}

\begin{abstract}
Continuing our study of spectral triples on quantum domains, we look at unbounded invariant and covariant derivations in the quantum annulus. In particular, we investigate whether such derivations can be implemented by operators with compact parametrices, a necessary condition in the definition of a spectral triple.
\end{abstract}

\maketitle
\section{Introduction}
Derivations in operator algebras are natural objects to investigate from the point of view of noncommutative geometry. In \cite{KMRSW1} we studied unbounded rotationally invariant and covariant derivations in the quantum disk.  We classified such derivations and looked at their implementations in various Hilbert spaces obtained from the GNS construction with respect to an invariant state. The question when such implementations are operators with compact parametrices was answered and, surprisingly, we found out that no implementation of a covariant derivation in any GNS Hilbert space for a faithful normal invariant state has compact parametrices for a large class of reasonable boundary conditions. 

This paper (and its prequel) arose from an attempt to construct geometric even spectral triples on quantum plane domains.  It is motivated by the classical example of the two-dimensional Dirac operator of the form:
\begin{equation*}
{\mathcal D} = \left[
\begin{array}{cc}
0 & D \\
D^* & 0
\end{array}\right],
\end{equation*}
where $D$ is the d-bar operator on a disk or an annulus in the complex plane, acting in $L^2$ space of the domain. Thus, it is very natural in the noncommutative case to look for $D$ as an implementation in some GNS Hilbert space of a covariant derivation in the underlying algebra, since there is no canonical analog of the d-bar operator, which is a covariant derivation, and since the choice of such $D$ would guarantee that the commutator condition of the spectral triple is satisfied. Unfortunately, the results in this paper show that spectral triples cannot be defined in such a way.  This is in contrast with classical analysis where for the Dirac operator above, subject to APS-like boundary conditions, the parametrices are compact. In noncommutative case however, no boundary conditions help. The key reason of the no-go theorem is the fundamental difference between classical differential operators and finite-difference operators with unbounded coefficients, the later appearing in quantum case. Nonetheless, since the assumptions of those constructions, while very natural, are not the most general possible, it remains an interesting challenge to construct meaningful geometric spectral triples for quantum domains in some other way.

Here we look at unbounded invariant and covariant derivations in the quantum annulus, which is the C$^*-$algebra generated by a special weighted bilateral shift.  This algebra was studied in detail in \cite{KM1} and \cite{KMRS}.  Much like the work presented in \cite{KMRSW1} we first classify such derivations and then look at their implementations in various Hilbert spaces obtained from the GNS construction with respect to an invariant state. However, even though the quantum annulus is singly generated, the classification is done differently than in \cite{KMRSW1} because the relations are more complicated and not so useful.
We then answer the question when such implementations are operators with compact parametrices and thus can be used to define spectral triples. Like the quantum disk situation, no implementation of a covariant derivation in any GNS Hilbert space for a faithful normal invariant state has compact parametrices for a large class of boundary conditions.  

Unlike the quantum disk however, the technical reasons are quite different for why this fails.  In \cite{KMRSW1}, we were able to study Fourier components of an implementation of a covariant derivation and show that their spectral properties contradict the implementation having compact parametrices.  The argument there followed similar calculations of the spectrum of the Ces{\`a}ro operator \cite{BHS}.    For the annulus the spectral argument for components is not conclusive. The coefficients, rather than polynomially bounded as in the disk case, now have to be rapidly decaying and rapidly increasing at different infinities. In one case the components of an implementation fail to have decaying norms while in the other case the implementing operator has an infinite multiplicity eigenvalue. Thus, even though the conclusions are similar in the two papers, the techniques are quite different.

Recall that if $A$ is a Banach algebra and if $\mathcal{A}$ is a dense subalgebra of $A $ then
a linear map $d:\mathcal{A}\to A $ is called a {\it derivation} if the Leibniz rule holds:
\begin{equation*}
d(ab) = ad(b) + d(a)b
\end{equation*}
for all $a,b\in\mathcal{A}$.

The following example is of the d-bar operator on the annulus, and it is the motivating example for the rest of the paper.
Let $A = C(A_r)$ be the C$^*$-algebra of continuous functions on the annulus: 
$$A_r:=\{z\in\C : r\leq z\leq 1\},$$ 
$0<r<1$.  If $\mathcal{A}$ is the algebra of polynomials in $z$ and $\bar z$ then\begin{equation*}
(da)(z) := \frac{\partial a(z)}{\partial \bar z}
\end{equation*}
is an unbounded, closable derivation in $A$.
Let $\rho_\theta: A\to A$ be the one-parameter family of automorphisms of $A$ given by the rotation $z\to e^{i\theta}z$ on the annulus. Notice that $\rho_\theta: \mathcal{A}\to\mathcal{A}$.
Moreover, $d$ is a {\it covariant} derivation in $A$ in the sense that it satisfies:
\begin{equation*}
d(\rho_\theta(a))= e^{-i\theta}\rho_\theta(d(a)),\ \ a\in\mathcal{A}.
\end{equation*}

The map $\tau: A\to \C$ given by \begin{equation*}
\tau(a)=\frac{1}{\pi(1-r^2)}\int_{A_r} a(z)d^2z,
\end{equation*}
is a $\rho_\theta$-invariant, faithful state on $A$.  The GNS Hilbert space $H_\tau$, obtained using the state $\tau$, is naturally identified with $L^2(A_r,d^2z)$, the completion of $A$ with respect to the usual inner product: 
$$\|a\|^2_\tau=\tau(a^*a)=\frac{1}{\pi(1-r^2)}\int_{A_r}|a(z)|^2d^2z.$$

The representation $\pi_\tau: A\to B(H_\tau)$ is given by multiplication: 
$$\pi_\tau(a)f(z)=a(z)f(z).$$ 

The unitary operator: 
$$U_{\tau,\theta}f(z):=f(e^{i\theta}z)$$
in $H_\tau$ implements $\rho_\theta$ in the sense that:
\begin{equation*}
U_{\tau,\theta}\pi_\tau(a)U_{\tau,\theta}^{-1} = \pi_{\tau}(\rho_\theta(a)).
\end{equation*} 
Then the covariant operator:
\begin{equation*}
(D_{\tau} f)(z) = \frac{\partial f(z)}{\partial\bar z}
\end{equation*}
on domain $\mathcal{D}_{\tau}=\mathcal{A}\subset A\subset H_\tau$ is an implementation of $d$ in $H_\tau$, i.e.
$$ [D_{\tau}, \pi_\tau(a)] = \pi_\tau(d(a)),$$
 for all $a\in\mathcal{A}$. 
 
The operator $D_{\tau}$ has an infinite dimensional kernel, so the operator $(1+D_{\tau}^*D_{\tau})^{-1/2}$ is not compact. This is because the annulus is a manifold with boundary and we need to impose elliptic-type boundary conditions to make $D_{\tau}$ elliptic, so that it has compact parametrices. 

Denote by $D^{max}_\tau$ the closure of $D_{\tau}$, hence there are no boundary conditions on $D^{max}_\tau$. On the other hand, let $D^{min}_\tau$ be the closure of $D_{\tau}$ defined on $C_0^\infty(A_r)$. The cokernel of $D^{min}_\tau$ is infinite dimensional. There exist however closed operators $D_{\tau}$ with compact parametrices, such that: 
$$D^{min}_\tau\subset D_{\tau}\subset D^{max}_\tau;$$ 
as in the disk case, examples are given by Atiyah-Patodi-Singer (APS) type boundary conditions, see \cite{BBW}.  

A more general discussion of spectral triples for manifolds with boundary using operators with APS boundary conditions is contained in \cite{BS}. Additionally, recent paper \cite{FGMR} studies generalities of spectral triples on noncommutative manifolds with boundary.

The paper is organized as follows.  In section 2 we review the quantum annulus. Section 3 contains a classification of invariant and covariant derivations in the quantum annulus. In section 4 we classify invariant states on the quantum annulus and describe the corresponding GNS Hilbert spaces and representations. Section 5 describes implementations of those derivations in the GNS Hilbert spaces of section 4. Finally in section 6 we analyze when those implementations have compact parametrices.  

\section{Quantum Annulus}

Let $\{E_k\}$ be the canonical basis for $\ell^2(\Z)$ and let $U$ be the bilateral shift, i.e. 
$$UE_k = E_{k+1}.$$  
We use the diagonal label operator: 
$$\K E_{k} = kE_{k},$$ 
so that, for a bounded function $a: \Z \to\C$, we can write: 
$$a(\K) E_{k} = a(k)E_{k}.$$  
We have the following crucial commutation relation for a diagonal operator $a(\K)$:
\begin{equation}\label{CommRel}
a(\K)\,U = Ua(\K +1).
\end{equation}

Consider the following special weighted bilateral shift:
\begin{equation*}
U_rE_k = \left\{
\begin{aligned}
rE_{k+1} & \quad k<0 \\
E_{k+1} & \quad k\ge 0,
\end{aligned}\right.
\end{equation*}
$0<r<1$.
Let $A = C^*(U_r)$, the unital C$^*-$algebra generated by $U_r$.   Reference \cite{KM1} argues, with explicit details in \cite{KMRS}, that this algebra can be thought of as a quantum annulus, see also \cite{KL3}.   Its structure is described by the following short exact sequence:
\begin{equation*}
0\longrightarrow \mathcal{K} \longrightarrow A  \longrightarrow C(S^1)\oplus C(S^1) \longrightarrow 0,
\end{equation*}
where $\mathcal{K} $ is the ideal of compact operators in $\ell^2(\Z)$.  In fact $\mathcal{K}$ is the commutator ideal of the algebra $A$.   

The generators in this algebra are a bit difficult to work with, and we prefer to think of $A$ as a crossed product. To see how this works we need the following observation.

\begin{lem}\label{U_decomp_in_U_r}\begin{equation*}
U = \frac{1+r+r^2}{(1+r)r}U_r - \frac{1}{(1+r)r}U_rU_r^*U_r
\end{equation*}
Additionally, the orthogonal projection $P_0$ onto $E_0$ can be written in terms of $U_r$ and $U_r^*$ as:
\begin{equation*}
P_0 = \frac{1}{1-r^2}\left[U_r^*,U_r\right].
\end{equation*}
\end{lem}

\begin{proof}
This is a straightforward calculation, applying both sides of those identities to the basis elements.
\end{proof}

As a direct consequence we get the following formulas:
\begin{equation}\label{projections_in_U_r}
\begin{aligned}
&U^{-1} = U^* = \frac{1+r+r^2}{(1+r)r}U_r^* - \frac{1}{(1+r)r}U_r^*U_rU_r^* \\
&P_k = U^kP_0U^{-k} \\
&P_{\ge0} = \frac{1}{1-r^2}\left(U_r^*U_r - r^2I\right) \\
&P_{<0} = I - P_{\ge0},
\end{aligned}
\end{equation}
where $P_{\ge0}$ is the orthogonal projection onto $\textrm{span}\{E_k\}_{k\ge0}$, $P_{<0}$ is the orthogonal projection onto $\textrm{span}\{E_k\}_{k<0}$ and $P_k$ is the orthogonal projection onto $E_k$.

We call a function $a:\Z \to\C$ {\it eventually constant}, if there exists a natural number $k_0$ such that $a(k)$ are constants for $|k|\ge k_0$.  The smallest such $k_0$ is called the {\it domain constant}. The set of all such functions will be denoted by $c_{00}^\pm(\Z)$. 

Let $Pol(U_r, U_r^*)$ be the set of all polynomials in $U_r$ and $U_r^*$ and define
\begin{equation*}
\mathcal{A} = \left\{a = \sum_{n}U^n a_n(\K) \ : \ a(k)\in c_{00}^\pm(\Z),\textrm{ finite sum}\right\}.
\end{equation*}
We have the following observation.

\begin{prop}
$\mathcal{A} = Pol(U_r, U_r^*)$, the algebra of polynomials in generators $U_r$ and  $U_r^*$.
\end{prop}

\begin{proof}
It is clear that $\mathcal{A}$ is a $*$-subalgebra of $A$. We need to check $\mathcal{A}\subset Pol(U_r, U_r^*)$ and vice versa.  Notice that $U_r = Ua(\K)$ where $a(k) = r$ for $k<0$ and $a(k) = 1$ for $k\ge0$.  This function is obviously eventually constant so we have that $U_r\in\mathcal{A}$. Thus $Pol(U_r, U_r^*)\subset\mathcal{A}$.   

Next, by Lemma \ref{U_decomp_in_U_r} we have $U\in Pol(U_r, U_r^*)$.   Moreover we have $P_0$, $U^{-1}$, $P_k$, $P_{\ge0}$, $P_{<0}\in Pol(U_r, U_r^*)$, from equations (\ref{projections_in_U_r}).  Suppose that $a(k)\in c_{00}^\pm(\Z)$ with domain constant $L$ and set $a_- = a(k)$ for $k\le -L$ and $a_+ = a(k)$ for $k\ge L$.  We can decompose $a(\K)$ as follows:
\begin{equation}\label{diag_decomp}
a(\K) = a_-P_{\le -L} + \sum_{k=-L+1}^{L-1}a(k)P_k + a_+P_{\ge L}.
\end{equation}
Notice that we have: 
$$P_{\ge L} = P_{\ge0} - P_0 - P_1 -\ldots - P_{L-1}$$
and 
$$P_{\le -L} = P_{<0} - P_{-1} - P_{-2} - \ldots - P_{-L+1}.$$ 
It follows that $P_{\ge L}$ and $P_{\le -L}$ are both in $Pol(U_r, U_r^*)$ and consequently, by (\ref{diag_decomp}), we have $a(\K)\in Pol(U_r, U_r^*)$.  Thus it follows that $\mathcal{A}\subset Pol(U_r, U_r^*)$, completing the proof.
\end{proof}

Let $c(\Z)$ be the space of convergent sequences, and consider the abelian algebra: 
$$A_{diag} = \left\{a(\K)  : \{a(k)\}\in c(\Z) \right\}.$$   
It follows from the above proposition that this is precisely the subalgebra of all diagonal operators in $A$ and we have: 
$$A = C^*(U_r)=C^*(A_{diag}, U).$$
Additionally, because of formula (\ref{CommRel}), we can view the quantum annulus as the group crossed product of $A_{diag}$ with $\Z$ acting on $A_{diag}$  via shifts (translation by $n\in\Z$), that is: \begin{equation*}
A = A_{diag} \rtimes_{shift}\Z.
\end{equation*}
Also, we see that $(A,  A_{diag})$ is a Cartan pair \cite{R}.

\section{Derivations in the Quantum Annulus}

\subsection{Classification of derivations}

For each $\theta\in[0,2\pi)$, let $\rho_\theta : A\to A$ be an automorphism defined by $\rho_\theta(a(\K)) = a(\K)$ for a diagonal operator $a(\K)$, $\rho_\theta(U) = e^{i\theta}U$ and $\rho_\theta(U^*) = e^{-i\theta}U^*$. It is well defined on all of $A$ because it preserves the relations $U^*U=UU^*=I$ as well as (\ref{CommRel}). Alternatively, the action of $\rho_\theta$ can be written down using the label operator $\K$ as:
\begin{equation*}
\rho_\theta(a)=e^{i\theta\K}ae^{-i\theta\K}.
\end{equation*}
It follows that  $\rho_\theta : \mathcal{A}\to \mathcal{A}$.

Any derivation $d:\mathcal{A}\to A$ that satisfies the relation: 
$$\rho_\theta(d(a)) = d(\rho_{\theta}(a))$$ 
will be referred to as a $\rho_\theta$-{\it invariant} derivation. Similarly, any derivation $d:\mathcal{A}\to A$ that satisfies the relation  
$$d(\rho_\theta(a))= e^{-i\theta}\rho_\theta(d(a))$$ for all $a\in \mathcal{A}$ will be referred to as a $\rho_\theta$-{\it covariant} derivation.  

Notice that we have the identifications:
$$\mathcal{A}_{diag}:=\left\{a(\K) \ : \{a(k)\}\in c_{00}^\pm(\Z) \right\}=\{a\in\mathcal{A} : \rho_\theta(a) = a\},$$  
and similarly
$$A_{diag}= \left\{a(\K) \ : \{a(k)\}\in c(\Z) \right\}=\{a\in A : \rho_\theta(a) = a\}.$$

Additionally, we need the following sets and the identifications: 
$$\mathcal{A}_{cov} := \{a\in\mathcal{A} : \rho_\theta(a) = e^{i\theta}a\}=U\mathcal{A}_{diag},$$ and 
$$A_{cov} := \{a\in A : \rho_\theta(a) = e^{i\theta}a\}=UA_{diag}.$$

Below we use the following terminology: we say that a function $\beta:\Z \to\C$ has {\it convergent increments}, if the sequence of differences $\{\beta(k)-\beta(k-1)\}$ is convergent i.e. 
$$\{\beta(k)-\beta(k-1)\}\in c(\Z).$$  
The set of all such functions will be denoted by $c_{inc}(\Z)$. Similarly the set
of {\it eventually linear} functions is the set of $\beta:\Z \to\C$ such that: 
$$\{\beta(k)-\beta(k-1)\}\in c_{00}^\pm(\Z).$$

The following two propositions classify all invariant and covariant derivations $d:\mathcal{A}\to A$.  

\begin{prop}\label{invar_der_rep}
If $d$ is a $\rho_\theta$-invariant derivation $d:\mathcal{A}\to A$, then there exists a function $\beta\in c_{inc}(\Z)$, which is unique modulo an additive constant, such that $$d(a) = [\beta(\K), a]$$ for $a\in\mathcal{A}$.  If $d:\mathcal{A}\to \mathcal{A}$, then the corresponding
function $\beta(k)$ is eventually linear.
\end{prop}

\begin{proof}
Since $d$ is $\rho_\theta$-invariant we have that $d:\mathcal{A}_{diag}\to A_{diag}$, because 
$$\rho_\theta(d(a(\K)) = d(\rho_\theta(a(\K))) = d(a(\K)).$$ 
If $P\in \mathcal{A}_{diag}$ is a projection then, applying $d$ to $P^2=P$ gives:
\begin{equation*}
d(P)=d(P)P+Pd(P),
\end{equation*}
or, equivalently, using commutativity of $A_{diag}$:
\begin{equation*}
(2P-1)d(P)=0.
\end{equation*}
This implies that $d(P)=0$, and in fact  that $d(a(\K))=0$ for every $a\in \mathcal{A}_{diag}$ because $\mathcal{A}_{diag}$ is linearly generated by projections.

Next, consider the formula: 
$$\rho_\theta(d(U)) = d(\rho_\theta(U)) = e^{i\theta}d(U).$$ 
It says that $d(U)\in A_{cov}$, and thus we must have $d(U) = Ua(\K)$ for some $a(\K)\in A_{diag}$ by the above identifications.   Next we define $\beta(k)$ so we can write: 
$$a(\K) = \beta(\K+1) - \beta(\K).$$
In fact, the above formula determines $\beta(\K)$ uniquely up to a constant.  Moreover if $a(\K)\in\mathcal{A}_{diag}$, then $a(k)$ is eventually constant and so $\beta(k)$ must be eventually linear.  Then the commutation relation (\ref{CommRel}) implies that:
\begin{equation*}
d(U) = U(\beta(\K+1) - \beta(\K)) = \beta(\K)U - U\beta(\K) = [\beta(\K), U].
\end{equation*}
Using inductive reasoning we obtain: 
$$d\left(U^n\right)  =\left[\beta(\K), U^n\right],$$ 
then, using linearity and the decomposition of any $a\in\mathcal{A}$, we obtain the result.
\end{proof}

\begin{prop}\label{covar_der_rep}
If $d$ is a $\rho_\theta$-covariant derivation $d:\mathcal{A}\to A$, then there exists a unique function $\beta\in c_{inc}(\Z)$, such that  $d(a) = [U\beta(\K),a]$ for all $a\in\mathcal{A}$.  If $d:\mathcal{A}\to \mathcal{A}$ then the corresponding function $\beta(k)$ is eventually linear.
\end{prop}

\begin{proof}
Consider the following equation:
\begin{equation*}
\rho_\theta(d(U^*)) = e^{i\theta}d(\rho_\theta(U^*)) = e^{i\theta}e^{-i\theta}d(U^*) = d(U^*),
\end{equation*}
which says that $d(U^*)\in A_{diag}$. Consequently, this implies that $d(U^*) = \alpha(\K)$
for some $\alpha(\K)\in A_{diag}$, which then gives:
\begin{equation}\label{donU}
d(U)=-U^2\alpha(\K+1).
\end{equation}
 
Next we define a linear map $\tilde{d}:\mathcal{A}\to A$ via the formula:
$$\tilde{d}(a) := U^{-1}d(a).$$  
Notice that we have\begin{equation*}
\rho_\theta(\tilde{d}(a)) = \rho_\theta(U^{-1})\rho_\theta(d(a)) = e^{-i\theta}U^{-1}e^{i\theta}d(\rho_\theta(a)) = \tilde{d}(\rho_\theta(a)),
\end{equation*}
which shows that $\tilde{d}$ is $\rho_\theta$-invariant.
In particular this implies that $\tilde{d}:\mathcal{A}_{diag}\to A_{diag}$. Additionally, $\tilde{d}$ satisfies a twisted Leibniz rule of the form:
\begin{equation*}
\tilde{d}(a(\K)b(\K)) = \tilde{d}(a(\K))b(\K) + a(\K+1)\tilde{d}(b(\K)).
\end{equation*}

We claim that there is a diagonal operator $\beta(\K)$ such that: 
\begin{equation}\label{ddiag}
\tilde{d}(a(\K)) = \beta(\K)(a(\K) - a(\K+1)).
\end{equation}
To construct $\beta(k)$ we apply $\tilde{d}$ to both sides of the equality $P_{\ge k}^2 = P_{\ge k}$ and obtain: 
\begin{equation*}
\tilde{d}(P_{\ge k}) = \tilde{d}(P_{\ge k})P_{\ge k} + P_{\ge k-1}\tilde{d}(P_{\ge k}),
\end{equation*}
or, equivalently,\begin{equation*}
P_{\le k-1}\tilde{d}(P_{\ge k}) = P_{\ge k-1}\tilde{d}(P_{\ge k}).
\end{equation*}
This implies that $\tilde{d}(P_{\ge k})(l)$ is zero unless $l=k-1$. Define the sequence $\beta(k)$ by:
$$\beta(k-1)=-\tilde{d}(P_{\ge k})(k-1).$$ 
As a consequence we have:
\begin{equation*}
\tilde{d}(P_{\ge k})(\K)=-\beta(\K)P_{k-1}(\K)=\beta(\K)(P_{\ge k}(\K)-P_{\ge k-1}(\K))=\beta(\K)(P_{\ge k}(\K)-P_{\ge k}(\K+1)).
\end{equation*}
Using the decomposition of $a(\K)$ as a sum of orthogonal projections $P_{\ge k}$ we see that (\ref{ddiag}) holds for any $a(\K)\in A_{diag}$.  
Also, using the relation $\tilde{d}(a) = U^{-1}d(a)$, we get
\begin{equation*}
d(a(\K)) = U\tilde{d}(a(\K)) = U\beta(\K)(a(\K) - a(\K+1)) = [U\beta(\K), a(\K)].
\end{equation*}

Now we need to relate $\alpha$ and $\beta$. Applying $d$ to the commutation relation (\ref{CommRel}) and using (\ref{donU}) and  (\ref{ddiag}) yields:
\begin{equation*}
[U\beta(\K), a(\K)]U - a(\K)U^2\alpha(\K+1) = - U^2\alpha(\K+1)a(\K+1) + U[U\beta(\K), a(\K+1)],
\end{equation*}
implying that for all $a\in\mathcal{A}_{diag}$ we have:
\begin{equation*}
(\beta(\K+1) - \beta(\K) + \alpha(\K+1))(a(\K+2) - a(\K+1)) = 0 .
\end{equation*}
We conclude that $\alpha(\K) = \beta(\K-1) - \beta(\K)$, thus \begin{equation*}
d(U)=-U^2\alpha(\K+1)=[U\beta(\K),U].
\end{equation*}
Since $\alpha\in c(\Z)$ and $\alpha(k) = \beta(k-1) - \beta(k)$ we must have $\beta\in c_{inc}(\Z)$.   The result follows from the standard inductive reasoning in powers of $U$ and linearity.
\end{proof}

\subsection{Structure of derivations}

In \cite{KMRSW1} we studied when invariant derivations in the quantum disk are approximately bounded/approximately inner.  This was motivated by the questions arising in \cite{BEJ}, \cite{H}, and \cite{J}. We continue this discussion for the quantum annulus.  

Recall that $d$ is called {\it approximately inner} if there are $a_n\in A$  such that 
$$d(a) = \lim_{n\to\infty}[a_n,a]$$ 
for $a\in\mathcal{A}$.  If  $d(a) = \lim_{n\to\infty} d_n(a)$ for bounded derivations $d_n$  in $A$ then $d$ is called {\it approximately bounded}. Note also that any bounded derivation $d$ on $A$ can be written as a commutator $d(a)=[a,x]$ with $x$ in a weak closure of $A$, see \cite{KR}, \cite{S}. Recall also, from the proof of Proposition \ref{invar_der_rep}, that if $d:\mathcal{A}\to A$ is an invariant derivation then 
$$d(a(\K))=0$$ 
for every $a(\K)\in A_{diag}$ and so $d$ is completely determined by its value on $U$.

\begin{lem}
Let $d$ be a $\rho_\theta$-invariant derivation in $A$ with domain $\mathcal{A}$.  If $d$ is approximately bounded then there exists a sequence $\{\mu_n(k)\}\in\ell^\infty(\Z)$ such that 
\begin{equation*}
d(a) = \lim_{n\to\infty}[\mu_n(\K),a]
\end{equation*}
for all $a\in\mathcal{A}$, i.e. $d$ can be approximated by bounded $\rho_\theta$-invariant derivations.
\end{lem}

\begin{proof}
Given an element $a\in A$ we define its $\rho_\theta$-average $a_{av} \in A$ by:
\begin{equation*}
a_{av} := \frac{1}{2\pi}\int_0^{2\pi} \rho_\theta(a)\ d\theta .
\end{equation*}
It follows that $a_{av}$ is $\rho_\theta$-invariant since the Lebesgue measure $d\theta$ is translation invariant.  Additionally, all $\rho_\theta$-invariant operators in $\ell^2(\Z)$ are diagonal with respect to the basis $\{E_k\}$ so that $a_{av}\in A_{diag}$.

Since by assumption $d$ is approximately bounded, there exists a sequence of bounded operators $b_n$ such that: 
$$d(a) = \lim_{n\to\infty}[b_n,a]$$ 
for all $a\in\mathcal{A}$.  It suffices to show that: 
\begin{equation}\label{diag_conv}
\lim_{n\to\infty}[(b_n)_{av},a] = d(a),
\end{equation}
since $(b_n)_{av}$ is $\rho_\theta$-invariant for every $\theta$ and hence by Proposition \ref{invar_der_rep} it is given by the commutator with a diagonal operator $\mu_n(\K)$ such that $\{\mu_n(k)\}\in\ell^\infty(\Z)$ because of the assumption of boundedness.   

By the remark before the statement of the lemma it is enough to verify (\ref{diag_conv}) on the generator $U$. This computation is essentially identical to the similar one in the corresponding proof for the quantum disk in \cite{KMRSW1}.   This completes the proof.
\end{proof}

The following results answer the question when is a $\rho_\theta$-invariant derivation in $d:\mathcal{A}\to A$ approximately inner/bounded. They are stated here without a proof because the only verifications needed are for the action of the derivations on  $U$, and those are exactly the same as in the proofs in \cite{KMRSW1}.

\begin{prop}
Let $d(a) = [\beta(\K), a]$ be a $\rho_\theta$-invariant derivation in $A$ with domain $\mathcal{A}$.  If $d$ is approximately bounded then $\{\beta(k)-\beta(k-1)\}\in c_0(\Z)$, the space of sequences convergent to zero.
\end{prop}

We also have the following converse result.

\begin{prop}
If $d(a) = [\beta(\K), a]$ is a $\rho_\theta$-invariant derivation in $A$ with domain $\mathcal{A}$ such that $\{\beta(k) - \beta(k-1)\}\in c_0(\Z)$, then $d$ is approximately inner.
\end{prop}

\section{Invariant States}

\subsection{Classification of states}

Next we describe all the invariant states on $A$.  If $\tau : A\to\C$ is a state, then $\tau$ is called a $\rho_\theta$-invariant state on $A$ if it satisfies: 
$$\tau(\rho_\theta(a)) = \tau(a).$$ 
The results in this section closely parallel similar analysis in \cite{KMRSW1}.

Since $A = A_{diag} \rtimes_{shift}\Z$, there is a natural expectation $E : A\to A_{diag}$, i.e. $E$ is positive, unital and idempotent linear map
defined, for $a\in\mathcal{A}$, by the formula:
\begin{equation*}
E(a) = E\left(\sum_{n}U^n a_n(\K)\right) = a_0(\K)
\end{equation*}
and $a_0(\K)\in A_{diag}$.      Since $A_{diag}$ is the fixed point algebra for $\rho_\theta$, we immediately obtain the following lemma:

\begin{lem}
Suppose $\tau : A\to\C$ is a $\rho_\theta$-invariant state on $A$, then there exists a state $t : A_{diag}\to\C$ such that $\tau(a) = t(E(a))$ where $E$ is the natural expectation.   Conversely given a state $t : A_{diag}\to\C$, the formula $\tau(a) = t(E(a))$ defines a $\rho_\theta$-invariant state $\tau$ on $A$.
\end{lem}

To parametrize all invariant states we need to first identify the pure states.
\begin{lem}\label{pure_lem}
The pure states on $A_{diag}$, denoted by $t_k$ for $k\in\N$ and $t_{\pm\infty}$, are given by:
\begin{equation*}
\begin{aligned}
&t_k(a(\K)) = a(k) = \langle E_k, aE_k\rangle ,\\
&t_{\pm\infty}(a(\K)) = \lim_{k\to\pm\infty} a(k) = \lim_{k\to\pm\infty} t_k(a(\K)).
\end{aligned}
\end{equation*}
\end{lem}

\begin{proof}
$A_{diag}$ is a commutative $C^*-$algebra that is isomorphic to the algebra of continuous functions on the two-point compactification of $\Z$, that is\begin{equation*}
A_{diag} \cong C(\{-\infty\}\cup\Z\cup\{\infty\}).
\end{equation*}
So by general theory, see \cite{KR} for details, the pure states are the Dirac measures (or point mass measures).
\end{proof}

We have the following classification theorem of the $\rho_\theta$-invariant states on $A$.

\begin{prop}
The $\rho_\theta$-invariant states on $A$ are in the closed convex hull of the states $\tau_k$ and $\tau_{\pm\infty}$ where $\tau_k(a) = t_k(E(a))$ and $\tau_{\pm\infty}(a) = t_{\pm\infty}(E(a))$.  Explicitly, if $\tau$ is a $\rho_\theta$-invariant state, there exists weights $w(k)\ge0$ satisfying: 
$$\sum_{k\in\Z} w(k) = 1,$$ 
and non-negative numbers $\lambda_0$ and $\lambda_{\pm\infty}$, with: 
$$\lambda_0 + \lambda_\infty + \lambda_{-\infty}  = 1,$$ 
such that: 
\begin{equation*}
\tau = \lambda_\infty\tau_\infty + \lambda_{-\infty}\tau_{-\infty} + \lambda_0\sum_{k\in\Z}w(k)\tau_k.
\end{equation*}
In the formula above we have:
$$\lambda_0 = \sum_{k\in\Z}\tau(P_k), \ w(k) = \lambda_0^{-1}\tau(P_k),$$
and also:
$$\lambda_{-\infty} = \tau(P_{\le0})-\sum_{k\le0}\tau(P_k),\ \lambda_\infty = \tau(P_{\ge0})-\sum_{k\ge0}\tau(P_k),$$ 
where again $P_k$ is the orthogonal projection onto the one dimensional subspace spanned by $E_k$ and $P_{\ge0}$ is the orthogonal projection onto $\textrm{span}\{E_k\}_{k\ge0}$.
\end{prop}
\begin{proof} This is a straightforward consequence of Lemma \ref{pure_lem}, for similar details see also \cite{KMRSW1}.
\end{proof}

Below we use the following notation with $w(k)\ge0$ such that $\sum_{k\in\Z} w(k) = 1$: 
$$\tau_w(a):= \textrm{tr}(w(\K)a)=\sum_{k\in\Z}w(k)\tau_k(a).$$

\subsection{GNS representations}

Given a state $\tau$ on $A$ let $H_\tau$ be the GNS Hilbert space and let $\pi_\tau:A\to B(H_\tau)$ be the corresponding representation.
 We describe the Hilbert spaces and representations coming from the following  $\rho_\theta$-invariant states: $\tau_w$ with all $w(k)\ne 0$, $\tau_0$, and $\tau_{\pm\infty}$. The states $\tau_w$ with all $w(k)\ne 0$ are general $\rho_\theta$-invariant faithful normal states on $A$.

\begin{prop}\label{GNS_prop}
The  GNS Hilbert spaces with respect to the $\rho_\theta$-invariant states $\tau_w$ with all $w(k)\ne 0$, $\tau_0$, and $\tau_{\pm\infty}$ can be naturally identified with the following Hilbert spaces, respectively:
\begin{enumerate}
\item $H_{\tau_w}$ is the Hilbert space whose elements are power series \begin{equation*}
f = \sum_{n\in\Z}U^n f_n(\K)
\end{equation*}
such that \begin{equation}\label{w_inner_prod}
\|f\|_{\tau_w}^2 =\tau_w(f^*f)=\sum_{n\in\Z} \sum_{k\in\Z} w(k)|f_n(k)|^2<\infty.
\end{equation}
\item $H_{\tau_0} \cong \ell^2(\Z)$, $\pi_{\tau_0}$ is the defining representation of $A$.
\item $H_{\tau_{\pm\infty}} \cong L^2(S^1)$, $\pi_{\tau_{\pm\infty}}(U)$ is the multiplication by $e^{ix}$, and $\pi_{\tau_{\pm\infty}}(a(\K))$ is the multiplication by the constant $\lim_{k\to\pm\infty} a(k)$.
\end{enumerate}
\end{prop}

\begin{proof}
The first Hilbert space is just the completion of $A$ with respect to the inner product given by (\ref{w_inner_prod}) and was introduced in \cite{KM1}.  It is the natural analog of the classical space of square-integrable functions $L^2(A_r)$ for the quantum annulus.

The Hilbert space $H_{\tau_0}$ comes from the state $\tau_0(a) = \langle E_0, aE_0\rangle$. To describe it we first need to find the algebra: 
$$\mathcal{A}_{\tau_0}  = \{a\in\mathcal{A} : \tau_0(a^*a) = 0\}.$$  
A simple calculation yields:
\begin{equation*}
\tau_0(a^*a) = \sum_n  |a_n(0)|^2,
\end{equation*}
thus, if $\tau_0(a^*a) = 0$, we get that $a_n(0) = 0$ for all $n\in\Z$.    Then, we have:
\begin{equation*}
\mathcal{A}/\mathcal{A}_{\tau_0}  \cong \left\{a = \sum_n U^n a_n(0)P_0\right\},
\end{equation*}
and let $\|a\|_{\tau_0}^2 := \tau_0(a^*a)$.  Using the basis $\{E_n := U^nP_0\}$ in $\mathcal{A}/\mathcal{A}_{\tau_0}$ for $n\in\Z$, we can naturally identify $\mathcal{A}/\mathcal{A}_{\tau_0} $ with a dense subspace of $\ell^2(\Z)$.  

It is easy to describe the representation $\pi_{\tau_0} : A \to B(H_{\tau_0})$ of $A$ in the bounded operators on $H_{\tau_0}$. We have: 
$$\pi_{\tau_0}(U)E_n = E_{n+1},$$
 and 
\begin{equation*}
\pi_{\tau_0}(a(\K))E_n = a(\K)U^nP_0 = U^na(\K+n)P_0 = U^na(n)P_0 =a(n)E_n.
\end{equation*}
Notice also that we have:
$$\mathcal{A}/\mathcal{A}_{\tau_0}  \ni [I] \mapsto P_0:=E_0.$$ 
In other words, $\pi_{\tau_0}$ is the defining representation of the algebra $A$.

Next we look at the GNS spaces associated with $\tau_{\pm\infty}(a) = \lim_{k\to\pm\infty}\langle E_k , aE_k\rangle$.  If $a(\K)\in\mathcal{A}$,  we set 
$$a_{\pm\infty} = \lim_{k\to\pm\infty} a(k).$$ 
Again we want to find the subalgebra $\mathcal{A}_{\tau_{\pm\infty}}$ of $a\in\mathcal{A}$ such that $\tau_{\pm\infty}(a^*a) = 0$.  A direct computation shows that:
\begin{equation*}
\tau_\infty(a^*a) = \sum_n |a_{n,\pm\infty}|^2,
\end{equation*}
so $\tau_{\pm\infty}(a^*a) = 0$ if and only if $a_{n,\pm\infty} = 0$ for all $n$.  Now $\mathcal{A}/\mathcal{A}_{\tau_{\pm\infty}} $ can be identified with a dense subspace of $L^2(S^1)$ by
\begin{equation*}
\mathcal{A}/\mathcal{A}_{\tau_{\pm\infty}}  \ni [a] = \left[a = \sum_n U^n a_{n,\pm\infty}\right]\mapsto \sum_n a_{n,\pm\infty}e^{inx} =: f_a(x)\in L^2(S^1).
\end{equation*}
We also have the formula:
\begin{equation*}
\tau_{\pm\infty}([a]) = \frac{1}{2\pi}\int_0^{2\pi}f_a(x)\ dx.
\end{equation*}
The representation $\pi_{\tau_{\pm\infty}} : A \to B(H_{\tau_{\pm\infty}})$ is clearly seen to be given by:
$$\pi_{\tau_{\pm\infty}}(U)f(x) = e^{ix}f(x),$$
and 
\begin{equation*}
\pi_{\tau_{\pm\infty}}(a(\K))f(x) = a_{\pm\infty} f(x).
\end{equation*}
This completes the proof.
\end{proof}

\section{Implementations of Derivations}
Let $H_{\tau}$ be the Hilbert space formed from the GNS construction on $A$ using a $\rho_\theta$-invariant state $\tau$ and let $\pi_{\tau} : A \to B(H_{\tau })$ be the representation of $A$ in the bounded operators on $H_{\tau }$ via left multiplication, that is $\pi_{\tau }(a)f = [af]$.   We have that $[A]\subset H_{\tau }$ is dense in $H_{\tau }$ and $[I]\in H_{\tau }$ is cyclic. Let $\mathcal{D}_{\tau }$ be the subspace of $H_\tau$ given by:
$$\mathcal{D}_{\tau } = \pi_{\tau }(\mathcal{A})\cdot[I].$$  
Then $\mathcal{D}_{\tau }$ is dense in $H_{\tau }$.  Define $U_{\tau ,\theta} : \mathcal{D}_\tau\to \mathcal{D}_\tau$ via 
$$U_{\tau ,\theta}[a] = [\rho_\theta(a)].$$  
Notice that for every $\theta$ the operator $U_{\tau ,\theta}$ extends to a unitary operator in $H_{\tau }$.   Moreover, by direct calculation we get:
\begin{equation*}
U_{\tau ,\theta}\pi_\tau(a)U_{\tau ,\theta}^{-1} = \pi_{\tau }(\rho_\theta(a)).
\end{equation*}
It follows from the definitions that $U_{\tau ,\theta}(\mathcal{D}_{\tau })\subset\mathcal{D}_{\tau }$ and $\pi_\tau(\mathcal{A})(\mathcal{D}_{\tau })\subset\mathcal{D}_{\tau }$.

An operator $D_{\tau} :\mathcal{D}_\tau \to H_\tau$ is called an {\it implementation} of a derivation $d:\mathcal{A}\to A$ if 
$$[D_{\tau}, \pi_\tau(a)] = \pi_\tau(d(a)).$$ 
In this section we construct implementations of the derivations in $A$ in the GNS Hilbert spaces: $H_{\tau_w}$, $H_{\tau_0}$, and $H_{\tau_{\pm\infty}}$. The results are completely analogous to those in \cite{KMRSW1}, however we present here a more general approach to implementation questions.

It should be noted that an implementation of a derivation may not exist.  In fact, we have the following useful criterion.

\begin{prop}
Let $A$ be a $C^*$-algebra and let $\pi:A\to B(\mathcal{H})$ be its representation in some Hilbert space $\mathcal{H}$.  Let $\mathcal{A}\subset A$ be a dense subalgebra of $A$.  Let $d:\mathcal{A}\to A$ be a derivation and suppose that $x_0$ is a cyclic vector for $\pi(A)$.  There is an operator $D : \pi(\mathcal{A})x_0\to \mathcal{H}$ implementing $d$, i.e. 
$$[D,\pi(a)](x) = \pi(d(a))(x)$$
for every $a\in\mathcal{A}$ and $x\in\pi(\mathcal{A})x_0$, if and only if there is $z\in\mathcal{H}$ such that:
\begin{equation}\label{z_condition}
\textrm{ if }\pi(a)x_0=0\textrm{ then }
\pi(d(a))x_0 = -\pi(a)z.
\end{equation}
Moreover, any implementation $D$ of $d$ is of the form:
\begin{equation}\label{D_imp}
D(\pi(a)x_0) = \pi(d(a))x_0 + \pi(a)D(x_0),
\end{equation}
for $z:=D(x_0)\in\mathcal{H}$ satisfying the condition \ref{z_condition} above.
\end{prop}

\begin{proof}
We first prove the sufficient condition.  If there exists a $D$ such that $[D,\pi(a)] = \pi(da)$, then in particular we have $[D,\pi(a)]x_0 = \pi(da)x_0$.  Suppose that $\pi(a)x_0=0$, then
\begin{equation*}
\pi(da)x_0 = D(\pi(a)x_0) - \pi(a)D(x_0) = -\pi(a)D(x_0).
\end{equation*}
with $z= D(x_0)$.  This proves the sufficiency.

On the other hand, suppose there exists such a $z$ with the desired properties in the statement of the theorem.  Define $D$ by the formula above:
$$D(\pi(a)x_0) = \pi(da)x_0 + \pi(a)z.$$ 
We need to check that $D$ is well-defined.  If $\pi(a)x_0 = \pi(\tilde a)x_0$ then, by the properties of $z$, we have: 
$$\pi(d(\tilde a- a))x_0 = -\pi(\tilde a- a)z.$$
Consequently, this gives:
\begin{equation*}
\begin{aligned}
D(\pi(\tilde{a})x_0) &= \pi(d(\tilde{a}))x_0 + \pi(\tilde{a})z = \pi(d(\tilde{a}-a))x_0 + \pi(\tilde{a}-a)z + \pi(d(a))x_0 +\pi(a)z \\
&= -\pi(\tilde a- a)z + \pi(\tilde a- a)z + D(\pi(a)x_0) = D(\pi(a)x_0),
\end{aligned}
\end{equation*}
and so the definition of $D$ makes sense.  

Given $a$ and $b$ in $\mathcal{A}$, consider the following calculation:
\begin{equation*}
\begin{aligned}
\left[D,\pi(b)\right]\pi(a)x_0 &= D(\pi(b)\pi(a)x_0) - \pi(b)D(\pi(a)x_0) \\
&= d(\pi(ba))x_0 + \pi(b)\pi(a)z - \pi(b)d(\pi(a))x_0 - \pi(ba)z \\
&=d(\pi(b))\pi(a)x_0 + \pi(b)d(\pi(a))x_0 - \pi(b)d(\pi(a))x_0 = d(\pi(b))\pi(a)x_0.
\end{aligned}
\end{equation*}
It shows that $D$ is in fact an implementation of $d$, verifying the necessity and completing the proof.
\end{proof}

We want to apply the above proposition to our unital algebra $A$ and GNS representations $\pi_\tau$ coming from invariant states $\tau$. The cyclic vector $x_0$ can be chosen to be the class of the identity: $x_0=[I]$. Notice that, in condition (\ref{z_condition}) we have:
$$\{a\in\mathcal{A} : \pi_\tau(a)x_0=0\}  = \{a\in\mathcal{A} : \tau(a^*a) = 0\}=\mathcal{A}_{\tau},$$ 
and algebras $\mathcal{A}_{\tau}$ were described in proofs in the previous section.

\subsection{Invariant derivations}
We first consider implementations of $\rho_\theta$-invariant derivations. Let $d_\beta$ be a $\rho_\theta$-invariant derivation $d_\beta:\mathcal{A}\to A$, $d_\beta(a) = [\beta(\K), a]$, as described in Proposition \ref{invar_der_rep}. Since $\beta(k)$ is defined up to a constant, we normalize it so that $\beta(0)=0$.

\begin{defin}
$D_{\tau} :\mathcal{D}_\tau \to H_\tau$ is called an {\it invariant implementation} of a $\rho_\theta$-invariant derivation $d_\beta$ if 
$$[D_{\tau}, \pi_\tau(a)] = \pi_\tau(d_\beta(a))\ \textrm{ and }\ U_{\tau,\theta}D_{\tau} U_{\tau,\theta}^{-1} = D_{\tau}.$$
\end{defin}

Below we discuss invariant implementations of invariant derivations in $A$ in various GNS Hilbert spaces.

\begin{prop} There exists a function $\alpha(k)$, satisfying:
$$\sum_{k\in\Z}|\beta(k)-\alpha(k)|^2w(k)<\infty,$$ 
such that any invariant implementation $D_{\beta,\tau_w} :\mathcal{D}_{\tau_w} \to H_{\tau_w}$ of $d_\beta$ is given by:
\begin{equation*}
D_{\beta,\tau_w}a = \beta(\K)a-a\alpha(\K).
\end{equation*}
\end{prop}

\begin{proof}
We start by computing $U_{\tau_w,\theta}$. From the definitions we have:
\begin{equation*}
U_{\tau_w,\theta}(a) = \sum_n U^n e^{in\theta}a_n(\K).
\end{equation*}

Since $\tau_w$ is faithful we have that $\mathcal{A}_{\tau_w}=0$, which implies tha $z=D_{\beta,\tau_w}(I)$ is arbitrary, since both sides of condition (\ref{z_condition}) are equal to zero. However,
it follows from the assumptions of invariant implementations that $D_{\beta,\tau_w}(I)$ must be invariant with respect to $U_{\tau_w,\theta}$. This implies that $D_{\beta,\tau_w}(I) = \eta(\K)$ for some diagonal operator $\eta(\K)\in H_{\tau_w}$.  Thus the formula (\ref{D_imp}) gives:
\begin{equation*}
\begin{aligned}
D_{\beta,\tau_w}a = D_{\beta,\tau_w}\pi_{\tau_w}(a)\cdot I = [\beta(\K),a] + a\eta(\K) =  \beta(\K)a-a\alpha(\K),
\end{aligned}
\end{equation*}
where $\alpha(k)=\beta(k)-\eta(k)$. Notice also that $\eta(\K)\in H_{\tau_w}$ implies that:
\begin{equation*} 
\|\eta(\K)\|^2_{\tau_w}=\sum_{k\in\Z}|\beta(k)-\alpha(k)|^2w(k)<\infty.
\end{equation*}
Thus the result follows.
\end{proof}

\begin{prop}
There is a number $c$ such that any invariant implementation $D_{\beta,\tau_0} :\mathcal{D}_{\tau_0} \to \ell^2(\Z )$ of $d_\beta$ is of the form:
\begin{equation*}
D_{\beta,\tau_0} = c\cdot I + \beta(\K).
\end{equation*}
\end{prop}
\begin{proof}
First we find $U_{\tau_0,\theta}$. Since $\rho_\theta(U^nP_0)=e^{in\theta}U^nP_0$, we have:
$$U_{\tau_0,\theta}E_n=e^{in\theta}E_n.$$
Because $D_{\beta,\tau_0}E_0$ must be invariant with respect to $U_{\tau_0,\theta}$, we have 
$$D_{\beta,\tau_0}E_0 = cE_0$$ 
for some constant $c$.  

We need to check that $z=cE_0$ satisfies condition (\ref{z_condition}) to guarantee the existence of $D_{\beta,\tau_0}$.  In $\ell^2(\Z)$, $[I] = E_0$, therefore $\pi_{\tau_0}(a)E_0 = 0$ if and only if $a_n(0) = 0$ for every $n$ as in the proof of Proposition \ref{GNS_prop}. We have:
\begin{equation*}
\pi_{\tau_0}(d(a))E_0=\beta(\K)aE_0-a\beta(\K)E_0=-\beta(0)aE_0=0.
\end{equation*}
Consequently, $z$ has to satisfy $az=0$ whenever $aE_0=0$, which is clearly true for $z=cE_0$. 

By using the formula (\ref{D_imp}) we get:
\begin{equation*}
D_{\beta,\tau_0}E_n = [\beta(\K), U^n]E_0 + cE_n = (\beta(\K) - \beta(\K-n))E_n + cE_n = (\beta(n)  + c)E_n.
\end{equation*}
Notice that the above $D_{\beta,\tau_0}$ are the only implementations even if we don't assume their invariance. Indeed, consider any $z$ satisfying condition (\ref{z_condition}). Define the following operator:
\begin{equation*}
a_N = U^{-N}\delta_{N}(\K)
\end{equation*}
for $N\neq 0$, where $\delta_N(k)=1$ for $k=N$ and zero otherwise.  Clearly we have $a_NE_0 = 0$.  Consequently, we obtain:
$$0=\langle E_{0}, a_Nz\rangle=\langle E_{N}, z\rangle$$ for all $N\ne 0$, thus $z = cE_0$ for some constant $c$.   This completes the proof.
\end{proof}

\begin{prop}
There exists a number $c$ such that all invariant implementations $D_{\beta,\tau_{\pm\infty}} : \mathcal{D}_{\tau_{\pm\infty}} \to L^2(S^1)$ of $d_\beta$ are of the form:
\begin{equation*}
D_{\beta,\tau_{\pm\infty}} = \beta_{\pm\infty}\frac{1}{i}\frac{d}{dx} + c,
\end{equation*}
where
$$\beta_{\pm\infty}:=\lim_{k\to\pm\infty}\left(\beta(k)-\beta(k-1)\right).$$
\end{prop}
\begin{proof}
Like in the other proofs we need to understand what is the value of $D_{\beta,\tau_{\pm\infty}}$ on $[I]$, the constant function $1$.  A simple calculation shows that:
\begin{equation*} 
(U_{\tau_{\pm\infty},\theta} f)(x) = f(x + \theta).
\end{equation*}
It is clear by the invariance properties that there exists a constant $c$ such that: 
$$D_{\beta,\tau_{\pm\infty}}(1) = c\cdot 1.$$  

Notice that $\pi_{\tau_{\pm\infty}}(a) 1 = 0$ if and only if $a_{n,\pm\infty} = 0$.  Using the formula:
\begin{equation*}
d_\beta(a)=\sum_n U^n(\beta(\K+n)-\beta(\K))a_n(\K),
\end{equation*}
we see that if $\pi_{\tau_{\pm\infty}}(a) 1 = 0$ then also $\pi_{\tau_{\pm\infty}}(d_\beta(a)) 1 = 0$. Consequently, $z=c\cdot 1$ satisfies condition (\ref{z_condition}).

Notice that $\mathcal{D}_{\tau_{\pm\infty}}$ is the space of trigonometric polynomials on  $S^1$. By linearity we only need to look at $D_{\beta,\tau_{\pm\infty}}$ on $e^{inx}$.  We have:
\begin{equation*}
\begin{aligned}
D_{\beta,\tau_{\pm\infty}}(e^{inx}) &=  \pi_{\tau_{\pm\infty}}(d(U^n)) + \pi_{\tau_{\pm\infty}}(U^n)D_{\beta,\tau_{\pm\infty}}(1) \\
&= \pi_{\tau_{\pm\infty}}(U^n)\cdot\lim_{k\to\pm\infty}\left(\beta(k+n) - \beta(k)\right) + c\pi_{\tau_{\pm\infty}}(U^n)(1) \\
&= e^{inx}\left(n\beta_{\pm\infty} + c\right) = \beta_{\pm\infty}\frac{1}{i}\frac{d}{dx}(e^{inx}) + ce^{inx}.
\end{aligned}
\end{equation*}
This completes the proof.
\end{proof}

\subsection{Covariant derivations}

Now let $\tilde d_\beta$ be a $\rho_\theta$-covariant derivation $\tilde d_\beta:\mathcal{A}\to A$, of the form $\tilde d_\beta(a) = [U\beta(\K), a]$, as proved in Proposition \ref{covar_der_rep}.
Let $\tau$ be a $\rho_\theta$-invariant state.

\begin{defin}
$\tilde D_{\tau} :\mathcal{D}_\tau \to H_\tau$ is called an {\it implementation} of a $\rho_\theta$-covariant derivation $\tilde d_\beta$ if 
$$[\tilde D_{\tau}, \pi_\tau(a)] = \pi_\tau(\tilde d_\beta(a))\ \textrm{ and }\ U_{\tau,\theta}\tilde D_{\tau} U_{\tau,\theta}^{-1} = e^{i\theta}\tilde D_{\tau}$$
for every $a$ and $\theta$.
\end{defin}

We state without proofs the analogs of the above implementation results for covariant derivations; the verifications are simple modifications of the arguments for invariant derivations.

\begin{prop}
 There exists a function $\alpha(k)$, $\sum_{k\in\Z}|\beta(k)-\alpha(k)|^2w(k)<\infty$, such that any implementation $\tilde D_{\beta,\tau_w} :\mathcal{D}_{\tau_w} \to H_{\tau_w}$ of $\tilde d_\beta$ is uniquely represented by:
\begin{equation*}
\tilde D_{\beta,\tau_w}f = U\beta(\K)f - fU\alpha(\K).
\end{equation*}
\end{prop}

\begin{prop}
Any implementation $\tilde D_{\beta,\tau_0} :\mathcal{D}_{\tau_0} \to \ell^2(\Z )$ of $\tilde d_\beta$ is of the form:
\begin{equation*}
\tilde D_{\beta,\tau_0}a = U\beta(\K)a,
\end{equation*}
i.e. on basis elements $\tilde D_{\beta,\tau_0}E_n = \beta(n)E_{n+1}$.
\end{prop}

\begin{prop}
There exists a number $c$ such that any implementation $\tilde D_{\beta,\tau_{\pm\infty}} : \mathcal{D}_{\tau_{\pm\infty}} \to L^2(S^1)$ of $\tilde d_\beta$ is of the form:
\begin{equation*}
\tilde D_{\beta,\tau_\infty} = e^{ix}\left(\beta_{\pm\infty}\frac{1}{i}\frac{d}{dx} + c\right).
\end{equation*}
\end{prop}

\section{Compactness of Parametrices}
\subsection{Spectral triples}

We say that a closed operator $D$ has compact parametrices if the operators $(1+D^*D)^{-1/2}$ and $(1+DD^*)^{-1/2}$ are compact.
Other equivalent formulations were summarized in the appendix of \cite{KMRSW1}.  In what follows we analyze when do the implementations of the derivations in $A$  have compact parametrices. The cases below are very similar to the examples in \cite{KMRSW1}. Consequently, all those cases give even spectral triples for the algebra $A$. 

We use the same notation for the {\it closure} of the operators constructed in the previous section. 

\begin{prop}
The operators $D_{\beta,\tau_0}$, $\tilde D_{\beta,\tau_0}$ have compact parametrices if and only if $\beta(k)\to\infty$ as $|k|\to\infty$.
\end{prop}
\begin{proof} The operators $D_{\beta,\tau_0}$ are diagonal with eigenvalues $\beta(k)-\beta(0) + c$, which must go to infinity for the operators to have compact parametrices. The operators $\tilde D_{\beta,\tau_0}$ differ from the operators $D_{\beta,\tau_0}$ by a shift, so they behave in the same way.
\end{proof}

\begin{prop}
The operators $D_{\beta,\tau_\infty}$, $\tilde D_{\beta,\tau_\infty}$ have compact parametrices if and only if $\beta_{\pm\infty}\ne 0$.
\end{prop}
\begin{proof} Similarly to the proof of the proposition above, the operators $D_{\beta,\tau_\infty}$ are diagonal with eigenvalues $\pm\beta_{\pm\infty} n+c$, which go to infinity if and only if $\beta_{\pm\infty}\ne 0$.
\end{proof}

\begin{prop}
The operators $D_{\beta,\tau_w}$ have compact parametrices if and only if\begin{equation*}
\beta(k+n)-\alpha(k)\to\infty
\end{equation*}
as $|n|,|k|\,\to\infty$.
\end{prop}
\begin{proof}
The operators $D_{\beta,\tau_w}$ can be diagonalized using the Fourier series:
\begin{equation*}
f = \sum_{n\in\Z}U^n f_n(\K).
\end{equation*}
Computing $D_{\beta,\tau_w}f = \beta(\K)f-f\alpha(\K)$ we get:
\begin{equation*}
D_{\beta,\tau_w}f = \sum_{n\in\Z}U^n (\beta(\K+n)-\alpha(\K))f_n(\K).
\end{equation*}
It follows that the numbers $\beta(k+n)-\alpha(k)$ are the eigenvalues of the diagonal operator, and must diverge for the operator to have compact parametrices.
\end{proof}

\subsection{Covariant derivations and normal states}
Here we study the parametrices of the operators which implement  $\rho_\theta$-covariant derivations in GNS Hilbert spaces $H_{\tau_w}$ corresponding to faithful normal states. In this section, like in \cite{KMRSW1}, we use a different notation for $\tilde D_{\beta,\tau_w}$; we use instead:
\begin{equation*}
D_{\beta,\alpha,w}f = U\beta(\K)f - fU\alpha(\K),
\end{equation*}
a notation that indicates the coefficients of the operator.
Denote by  $D_{\beta,\alpha,w}^{max}$ the closure of $D_{\beta,\alpha,w}$ defined on $\mathcal{D}^{max}_{\tau }= \pi_{\tau }(\mathcal{A})\cdot[I]$.

Define the $*-$algebra:
\begin{equation*}
\mathcal{A}_0 = \left\{a = \sum U^n a_n(\K) \ : \ a_n(k)\in {c_{00}^\pm(\Z)}\right\},
\end{equation*}
where $c_{00}^\pm(\Z)$ are the sequences with compact support, i.e. eventually zero,
and let $D_{\beta,\alpha,w}^{min}$ be the closure of $D_{\beta,\alpha,w}$ defined on  $\mathcal{D}^{min}_{\tau } = \pi_{\tau }(\mathcal{A}_0)\cdot[I]$. 
Finally, we use the symbol $D_{\beta,\alpha,w}$ for any closed operator in $H_{\tau_w}$ such that:
$$D_{\beta,\alpha,w}^{min}\subset D_{\beta,\alpha,w}\subset D_{\beta,\alpha,w}^{max}.$$

The main objective of this section is to prove the following no-go result, analogous to the non-existence theorem in the case for the quantum disk, see \cite{KMRSW1}.

\begin{theo}\label{main_theo}
There is no closed operator $D_{\beta,\alpha,w}$ in $H_{\tau_w}$, $D_{\beta,\alpha,w}^{min}\subset D_{\beta,\alpha,w}\subset D_{\beta,\alpha,w}^{max}$, with $\beta_{\pm\infty}\ne 0$, such that $D_{\beta,\alpha,w}$ has compact parametrices.
\end{theo}

\begin{proof}
The proof proceeds through a number of steps, similar to those in \cite{KMRSW1}. Initially, by a sequence of equivalences, we show that the operator $D_{\beta,\alpha,w}$ has compact parametrices if and only if a simplified version of it has compact parametrices. Since in particular an operator with compact parametrices has to be Fredholm, the finiteness of kernel and cokernel, as well as closedness of the range, imply certain growth estimates on the parameters.  As mentioned before, these growth estimates on the parameters are very different from what occurred in the disk situation.  

First we show that the coefficients $\beta(k)$ can be replaced by its absolute values. We need the following information.

\begin{lem}
Let $\{\beta(k)\}$ be sequence of complex numbers.  If $\beta(k+1)-\beta(k)\to \beta_\infty$ as $k\to\infty$ and $\beta(k)-\beta(k+1)\to \beta_{-\infty}$ as $k\to-\infty$ with $\beta_{\pm\infty}\neq0$, then there exists positive constants $c_1$ and $c_2$, and a nonnegative constant $c_3$ such that:
\begin{equation*}
c_2(|k|+1)-c_3  \le |\beta(k)|\le c_1(|k|+1).
\end{equation*}
Moreover $\left||\beta(k+1)|-|\beta(k)|\right|$ is bounded.
\end{lem}

\begin{proof}
We prove first that: 
$$\beta(k) = \beta_{\pm\infty}(|k|+1)(1 + o(1))$$ as $k\to\pm\infty.$ 
From this the first inequality follows immediately.  We decompose $\beta(k)$ as follows:
\begin{equation*}
\beta(k)=\left\{
\begin{aligned}
&\beta_\infty\cdot(k+1)+\beta_0(k) &&\textrm{ for }k\ge 0\\
&\beta_{-\infty}\cdot k + \beta_1(k) &&\textrm{ for }k<0,
\end{aligned}\right.
\end{equation*}
so that 
$$\beta_0(k)-\beta_0(k-1)=:\psi_0(k)\to 0$$ 
as $k\to\infty$ and 
$$\beta_1(k)-\beta_1(k+1) =:\psi_1(k)\to 0$$ 
as $k\to-\infty$. We want to show that:
\begin{equation*}
\frac{\beta_0(k)}{|k|+1}\to 0 \quad\textrm{and}\quad \frac{\beta_1(k)}{|k|+1}\to 0
\end{equation*}
as $k\to\pm\infty$ respectively.  We study $k\ge0$ first, $k<0$ is similar.  Notice that we have:
 \begin{equation*}
\frac{\beta_0(k)}{|k|+1} = \frac{\beta_0(k)-\beta_0(-1)}{k+1} + \frac{\beta_0(-1)}{k+1} = \frac{1}{k+1}\sum_{j=0}^k\psi_0(j) + \frac{\beta_0(-1)}{k+1}.
\end{equation*}
 It is clear that the second term above goes to $0$ as $k\to\infty$. 
For the first term, given $\varepsilon>0$ we choose $j_\varepsilon$ so that $|\psi_0(j)|<\varepsilon$ for $j\ge j_\varepsilon$. Then we split the sum as follows:
\begin{equation*}
\frac{1}{k+1}\sum_{j=0}^k|\psi_0(j)|=\frac{1}{k+1}\sum_{j=0}^{j_\varepsilon -1}|\psi_0(j)| + \frac{1}{k+1}\sum_{j=j_\varepsilon}^k|\psi_0(j)| \le \frac{1}{k+1}\sum_{j=0}^{j_\varepsilon -1}|\psi_0(j)| + \varepsilon,
\end{equation*}
and then choose $k_\varepsilon$ so that: 
$$\frac{(\sup_l|\psi_0(l)|)\,j_\varepsilon}{k+1}< \varepsilon$$ 
for $k\ge k_\varepsilon$. Therefore it follows that: 
$$\frac{\beta_0(k)}{|k|+1}=o(1).$$

The second part of the lemma follows from the estimate:
\begin{equation*}
\left||\beta(k+1)| - |\beta(k)|\right|\le |\beta(k+1) - \beta(k)| < \infty.
\end{equation*}
\end{proof}

\begin{lem}
Under the assumption of Theorem \ref{main_theo} the operator $D_{\beta,\alpha,w}$ has compact parametrices if and only if the operator $D_{|\beta|,\alpha,w}$, satisfying: 
$$D_{|\beta|,\alpha,w}^{min}\subset D_{|\beta|,\alpha,w}\subset D_{|\beta|,\alpha,w}^{max},$$ 
has compact parametrices.
\end{lem}

\begin{proof}
Define the unitary operator $V(\K)$ by: 
$$V(k) = \textrm{exp}\left(i\sum_{j=0}^{k-1}\textrm{Arg}(\beta(j))\right),\ k\geq 1, \textrm{ and } V(k) = \textrm{exp}\left(-i\sum_{j=k}^{0}\textrm{Arg}(\beta(j))\right), \ k\leq 0,$$ 
and consider the map $\phi: f\mapsto V(\K)f$ for $f\in\mathcal{H}_{\tau_w}$. This map preserves the domains $\mathcal{D}^{min}_{\tau }$ and $\mathcal{D}^{max}_{\tau }$. A direct computation gives that: 
$$D_{|\beta|,\alpha,w} = \phi D_{\beta,\alpha,w}\phi^{-1}.$$  
This shows that $D_{\beta,\alpha,w}$ and $D_{|\beta|,\alpha,w}$ are unitarily equivalent, thus completing the proof.
\end{proof}

\begin{lem}\label{para_shift}
The operator $D_{\beta,\alpha,w}$  has compact parametrices if and only if for any bounded sequences $\gamma_1(k)$ and $\gamma_2(k)$ the operator $D_{\beta+\gamma_1,\alpha+\gamma_2,w}$ satisfying:
$$D_{\beta+\gamma_1,\alpha+\gamma_2,w}^{min}\subset D_{\beta+\gamma_1,\alpha+\gamma_2,w}\subset D_{\beta+\gamma_1,\alpha+\gamma_2,w}^{max},$$ 
has compact parametrices.
\end{lem}

\begin{proof}
Notice that the difference $D_{\beta+\gamma_1,\alpha+\gamma_2,w}-D_{\beta,\alpha,w}$ is bounded, hence the two operators both either have or do not have compact parametrices simultaneously, see the appendix of \cite{KMRSW1}.
\end{proof}

It follows from those lemmas that, without loss of generality, we may assume that $\beta(k)>0$, where $\beta(k)$ satisfies inequalities:
\begin{equation}\label{growth_cond}
\begin{aligned}
c_2(|k|+1)  \le \beta(k)&\le c_1(|k|+1),\\
|\beta(k+1) - \beta(k)| &< \infty,
\end{aligned}
\end{equation}
where $c_1$ and $c_2$ are positive constants.

Next we look at properties of the coefficient $\alpha(k)$. If $f$ is in the domain of the operator $D_{\beta,\alpha,w}$ and:
\begin{equation*}
f = \sum_n U^n f_n(\K),
\end{equation*}
we can write $D_{\beta,\alpha,w}f = U\beta(\K)f - fU\alpha(\K)$ in Fourier components as:
\begin{equation}\label{DFcomp}
D_{\beta,\alpha,w}f = \sum_n U^{n+1} (D_nf_n)(\K),
\end{equation}
where\begin{equation*}
(D_nf)(k) = \beta(k+n)f(k) - \alpha(k)f(k+1).
\end{equation*}

\begin{lem}
If the operator $D_{\beta,\alpha,w}$ has compact parametrices, then  $\alpha(k)$ has at most finitely many zeros.
\end{lem}

\begin{proof}
Suppose that $\alpha(k)$ has infinitely many zeros.  Let $k_0$ be one of the zeroes of $\alpha(k)$. Then the sequence $\delta_{k_0+1}(k)$ is an eigenvector of $D_n$ with eigenvalue $\beta(k_0+1+n)$, where $\delta_j(k)=1$ for $k=j$ and zero otherwise. Indeed, we have:
\begin{equation*}
\begin{aligned}
(D_n\delta_{k_0+1})(k) &= \beta(k+n)\delta_{k_0+1}(k) - \alpha(k)\delta_{k_0+1}(k+1)= \\
&=\beta(k_0+1+n)\delta_{k_0+1}(k) - \alpha(k_0)\delta_{k_0+1}(k+1)=
\beta(k_0+1+n)\delta_{k_0+1}(k) .
\end{aligned}
\end{equation*}
Using this observation we obtain that: 
\begin{equation*}
U^{-1}D_{\beta,\alpha,w}U^{-(k_0+1)}\delta_{k_0+1}(\K)=\beta(0)U^{-(k_0+1)}\delta_{k_0+1}(\K).
\end{equation*}
Notice that $U^{-(k_0+1)}\delta_{k_0+1}(\K)$ has finite support and hence belongs to the domain of $D_{\beta,\alpha,w}^{min}$, and so it is in the domain of any $D_{\beta,\alpha,w}$.    Thus, if $\alpha(k)$ has infinitely many zeros, the operator: 
$$U^{-1}D_{\beta,\alpha,w}-\beta(0)I$$ 
has an infinite dimensional kernel, which is impossible since it has compact parametrices and hence it is Fredholm by the results in the appendix of \cite{KMRSW1}.  Therefore $\alpha(k)$ can only have finitely many zeros.  This completes the proof.
\end{proof}

As a consequence of the above lemma, and also Lemma \ref{para_shift}, we assume from now on that $\alpha(k)\ne 0$ for every $k$. While that conclusion is the same as in \cite{KMRSW1}, the argument above is different.

It is more convenient to work in the Hilbert space with no weights. The corresponding equivalence is described by the following lemma.

\begin{lem}
Let $\mathcal{H}_{\tau_w}$ be the weighted Hilbert space of Proposition \ref{GNS_prop}, part 1, and let $\mathcal{H}$ be that Hilbert space for weight $w(k)=1$.  The operator $D_{\beta,\alpha,w}$ in $\mathcal{H}_{\tau_w}$ has compact parametrices if and only if the operator $D_{\beta,\tilde\alpha,1}$ in $\mathcal{H}$, also satisfying: $$D_{\beta,\tilde\alpha,1}^{min}\subset D_{\beta,\tilde\alpha,1}\subset D_{\beta,\tilde\alpha,1}^{max}$$ 
has compact parametrices,
where\begin{equation*}
\tilde{\alpha}(k) = \alpha(k)\frac{\sqrt{w(k)}}{\sqrt{w(k+1)}}.
\end{equation*}
\end{lem}

\begin{proof}
In $\mathcal{H}_{\tau_w}$ write the norm as:
\begin{equation*}
\|f\|_w^2 = \textrm{tr}(w(\K)f^*f) = \textrm{tr} \left(fw(\K)^{1/2}\right)^*\left(w(\K)^{1/2}f\right),
\end{equation*}
and set $\varphi(f) = fw(\K)^{1/2}:\mathcal{H}_{\tau_w}\to\mathcal{H}$.  Then $\varphi$ is an isomorphism of Hilbert spaces.    Moreover, we have:
\begin{equation*}
\varphi(D_{\beta,\alpha,w}\varphi^{-1}f) = U\beta(\K)f - fU\alpha(\K)\frac{\sqrt{w(\K)}}{\sqrt{w(\K+1)}}=D_{\beta,\tilde\alpha,1}f,
\end{equation*}
and $\varphi D_{\beta,\alpha,w}\varphi^{-1}$ is an unbounded operator in $\mathcal{H}$, and so $D_{\beta,\alpha,w}$ and $D_{\beta,\tilde\alpha,1}$ are unitarily equivalent, thus completing the proof. Notice also that $\tilde{\alpha}(k)\ne 0$, because ${\alpha}(k)\ne 0$.
\end{proof}

From now on we work with operators $D_{\beta,\alpha,1}^{min}\subset D_{\beta,\alpha,1}\subset D_{\beta,\alpha,1}^{max}$ in the unweighted Hilbert space $\mathcal{H}$. For convenience we define a sequence $\{\mu(k)\}$ such that 
\begin{equation*}
\alpha(k) = \beta(k)\frac{\mu(k+1)}{\mu(k)}.
\end{equation*}
The sequence $\mu$, normalized by $\mu(0)=1$,  is completely determined by the above equation in terms of $\alpha$ and $\beta$ and will be used as a coefficient instead of $\alpha$. Using $\mu(k)$ we rewrite the two main operators as follows:
\begin{equation*}
(D_nf)(k) = \beta(k+n)\left(f(k) - \frac{\beta(k)}{\beta(k+n)}\frac{\mu(k+1)}{\mu(k)}f(k+1)\right), 
\end{equation*}\begin{equation*}
(D_n)^*f(k) = {\beta}(k+n)\left(f(k) - \frac{{\beta}(k-1)}{{\beta}(k+n)}\frac{\overline{\mu}(k)}{\overline{\mu}(k-1)}f(k-1)\right). \end{equation*}

The next goal is to compute the kernel and the cokernel of $D_{\beta,\alpha,1}$. This is done by using Fourier decomposition (\ref{DFcomp}) and the operators $D_n$ and $(D_n)^*$ above.

\begin{lem}\label{formal_kernels}
The formal kernels of $D_n$ and $(D_n)^*$ are one dimensional and are spanned by, correspondingly:
\begin{equation*}
h_n(k) =\left\{
\begin{aligned}
&\frac{1}{\mu(k)}\prod_{j=0}^{n-1}\beta(j+k) &&\quad\textrm{for }n> 0\\
&\frac{1}{\mu(k)} &&\quad\textrm{for }n=0\\
&\frac{1}{\mu(k)}\prod_{j=n}^{-1}\frac{1}{\beta(j+k)} &&\quad\textrm{for }n<0, 
\end{aligned}\right.
\end{equation*}
and
\begin{equation*}
\tilde h_n(k) = \left\{
\begin{aligned}
&\overline{\mu}(k)\prod_{j=0}^n\frac{1}{{\beta}(j+k)} &&\quad\textrm{for }n> 0\\
&\overline{\mu}(k) &&\quad\textrm{for }n=0\\
&\overline{\mu}(k)\prod_{j=n+1}^{-1}{\beta}(j+k) &&\quad\textrm{for }n<0.
\end{aligned}\right. 
\end{equation*}
\end{lem}

\begin{proof}
This is a result of straightforward calculations similar to those in \cite{KM1} and \cite{KMRSW1}.
\end{proof}

The computations above were formal; to actually compute the kernel and the cokernel of $D_{\beta,\alpha,1}$ we need to look at only those solutions which are in the domain/codomain of $D_{\beta,\alpha,1}$. We have the following inclusions:
\begin{equation*}
\textrm{ker\,}D_{\beta,\alpha,1}^{min}\subset \textrm{ker\,}D_{\beta,\alpha,1}\subset \textrm{ker\,}D_{\beta,\alpha,1}^{max},
\end{equation*}
and\begin{equation*}
\textrm{coker\,}D_{\beta,\alpha,1}^{max}\subset \textrm{coker\,}D_{\beta,\alpha,1}\subset \textrm{coker\,}D_{\beta,\alpha,1}^{min}.
\end{equation*}

The following lemma exhibits the key departure from the analogous classical analysis of the d-bar operator.

\begin{lem}\label{kernel_property}
If the operator $D_{\beta,\alpha,1}$ has compact parametrices, then both $\textrm{ker\,}D_{\beta,\alpha,1}^{max}$ and $\textrm{coker\,}D_{\beta,\alpha,1}^{min}$ are trivial. Moreover, we have:
\begin{equation}\label{inf_sums}
\sum_{k\in\Z} |h_n(k)|^2 =\sum_{k\in\Z} |\tilde h_n(k)|^2 =\infty
\end{equation}
for every $n$, where $h_n$ and $\tilde h_n$ are the formal kernels from Lemma \ref{formal_kernels}.
\end{lem}

\begin{proof}
Let $h_n$ and $\tilde{h}_n$ be the formal solutions to the equations $D_nh_n = 0$ and $(D_n)^*\tilde h_n = 0$ respectively, as described in  Lemma \ref{formal_kernels}.   First we study the $\ell^2(\Z)$ kernel of $D_n$.  There are two possibilities: 

$(1)$  there exists $n_0\in\Z$ such that $\|h_{n_0}\|<\infty$, or

$(2)$ $\|h_{n}\|=\infty$ for all $n$.

Consider the  first case. It is clear from the growth conditions (\ref{growth_cond}) and the formulas for $h_n$ that if there exists $n_0\in\Z$ such that $\|h_{n_0}\|<\infty$ then $\|h_n\|<\infty$ for all $n\le n_0$. Those inequalities imply that $D_{\beta,\alpha,1}^{max}$ has an infinite dimensional kernel. We argue below that in this case the kernel of $D_{\beta,\alpha,1}^{min}$ is also infinite dimensional, in contrast to classical theory. 

For any $n<n_0$ consider the sequence:
\begin{equation*}
h_n^N(k) = \left\{
\begin{aligned}
&h_n(k) &&\textrm{ for }|k|\le N \\
&0 &&\textrm{ else.}
\end{aligned}\right.
\end{equation*}
Notice that, because it is eventually zero, the sequence $h_n^N(k)$ is in the domain of $D_n^{min}$ and also $h_n^N\to h_n$ in $\ell^2(\Z)$ as $N\to\infty$.   Moreover, a  key direct calculation shows that: 
\begin{equation*}
D_nh_n^N(k) = \left\{
\begin{aligned}
&\beta(n+N)h_n(N)= h_{n+1}(N) &&\textrm{ for }k=N \\
&-\beta(-N-1)\frac{\mu(-N)}{\mu(-N-1)}h_n(-N) = -h_{n+1}(-N-1) &&\textrm{ for }k = -N-1\\
&0 &&\textrm{ else.}
\end{aligned}\right.
\end{equation*}
From this we see that $D_nh_n^N\to0$ as $N\to\infty$ since $\|h_n\|<\infty$ for all $n\leq n_0$.  This shows that the formal kernel of $D_n$ is contained in the domain of $D_n^{min}$.  In turn, this implies that $D_{\beta,\alpha,1}$ has an infinite dimensional kernel, contradicting the fact that $D_{\beta,\alpha,1}$ is Fredholm.  

A similar argument produces an infinite dimensional cokernel for $D_{\beta,\alpha,1}$ assuming that there exists $n_0\in\Z$ such that $\|\tilde h_{n_0}\|<\infty$.  Consequently, this does not happen in our case, and $\|h_{n}\|=\|\tilde h_{n}\|=\infty$ for all $n$. 
But that means that the $\ell^2(\Z)$ kernels of $D_n$ and $(D_n)^*$ are all trivial for any $n$. This implies that both $\textrm{ker\,}D_{\beta,\alpha,1}^{max}$ and $\textrm{coker\,}D_{\beta,\alpha,1}^{min}$ are trivial. Thus the proof is complete.
\end{proof}

It follows from the above lemma, and from the ``sandwich" proposition in appendix of \cite{KMRSW1}, that all three operators $D_{\beta,\alpha,1}^{min}\subset D_{\beta,\alpha,1}\subset D_{\beta,\alpha,1}^{max}$ have compact parametrices.

Using the above result we show that existence of compact parametrices for $D_{\beta,\alpha,1}$ implies specific growth estimates on coefficients $\mu(k)$ in the following lemma. Notice that, while the methods used here are similar to those in \cite{KMRSW1}, the outcomes are quite different; instead of being polynomially increasing or decreasing the coefficients $\mu(k)$ now have to be rapidly increasing or decreasing. 
\begin{lem}\label{mu_ineq}
If the operator $D_{\beta,\alpha,1}$ such that $D_{\beta,\alpha,1}^{min}\subset D_{\beta,\alpha,1}\subset D_{\beta,\alpha,1}^{max}$ has compact parametrices then the Fourier components $(D_n)^{min}$ and $((D_n)^*)^{max}$ are invertible operators with bounded inverses. Moreover, for every $N>0$ there exists constants $C_N>0$ such that exactly one of the following two holds:

\begin{enumerate}
\item $|\mu(k)|\ge C_N(1+|k|)^{N}$ for $k>0$ and $|\mu(k)|\le C_N(1+|k|)^{-N}$ for $k<0$, or
\item $|\mu(k)|\le C_N(1+|k|)^{-N}$ for $k>0$ and $|\mu(k)|\ge C_N(1+|k|)^{N}$ for $k<0$.
\end{enumerate}
\end{lem}

\begin{proof}
The Fredholm property of $D_{\beta,\alpha,1}$ implies that the ranges of $D_{\beta,\alpha,1}^{max}$ and $((D_{\beta,\alpha,1})^*)^{max}$ are closed. By the proof of Lemma \ref{kernel_property}  the $\ell^2(\Z)$ kernels of $D_n$ and $(D_n)^*$ are trivial for all $n$.  It follows that we have:
$$\textrm{Ran }D_n^{max} = \textrm{Ran }(D_n^*)^{max} = \ell^2(\Z).$$ 
In particular this says that we have inclusions:
$$\eta_n:=(D_n^{max})^{-1}\chi_0(k)\in\ell^2(\Z)\ \textrm{ and }\ \tilde\eta_n:=((D_n^*)^{max})^{-1}\chi_0(k)\in\ell^2(\Z),$$
where, as before, $\chi_0(k) = 1$ for $k=0$ and zero for all other $k$.  

Since the right-hand sides of the  equations: 
\begin{equation}\label{eta_eq}
(D_n\eta_n)(k) = \chi_0(k)\ \textrm{ and }\ (D_n)^*\tilde\eta_n(k) = \chi_0(k)
\end{equation}
are zero when $k\neq0$, the solutions are multiples of the kernel solutions from Lemma \ref{formal_kernels}:
\begin{equation*}
\eta_n(k) = \left\{
\begin{aligned}
&c_1^+(n)h_n(k) &&\textrm{ for }k>0 \\
&c_1^-(n)h_n(k)  &&\textrm{ for }k<0,
\end{aligned}\right.\ \textrm{ and }\ 
\tilde\eta_n(k) = \left\{
\begin{aligned}
&c_2^+(n)\tilde h_n(k) &&\textrm{ for }k>0 \\
&c_2^-(n)\tilde h_n(k) &&\textrm{ for }k<0.
\end{aligned}\right.
\end{equation*}
When $k=0$ in equations (\ref{eta_eq}) the right-hand side is one and a simple calculation shows that the constants $c_1^\pm(n)$, $c_2^\pm(n)$ satisfy the following:
\begin{equation*}
\begin{aligned}
c_1^+(n) - c_1^-(n) &= \mu(0)\prod_{j=0}^n\frac{1}{\beta(j)} &&\textrm{ for }n\ge 0, \\
c_1^+(n) - c_1^-(n) &= \mu(0)\prod_{j=n+1}^{-1}\beta(j) &&\textrm{ for }n<0, \\
c_2^+(n) - c_2^-(n) &= \frac{1}{\overline{\mu}(0)}\prod_{j=0}^{n-1}\beta(j) &&\textrm{ for }n\ge 0, \\
c_2^+(n) - c_2^-(n) &= \frac{1}{\overline{\mu}(0)}\prod_{j=n}^{-1}\beta(j) &&\textrm{ for }n<0.
\end{aligned}
\end{equation*}
In particular for any fixed $n$ the constants $c_1^\pm(n)$ cannot be both equal to zero, and the same for $c_2^\pm(n)$.

Next, from Lemma \ref{kernel_property} we know that the two-sided sums from formula (\ref{inf_sums}) are infinite. This can only be reconciled with the expressions for $\eta_n$ and $\tilde\eta_n$, which are in $\ell^2(\Z)$, if some one-sided sums of squares of $\eta_n$ and $\tilde\eta_n$ are finite and some coefficients $c_1^\pm(n)$, $c_2^\pm(n)$ are zero.
In fact we will show below that there are only two options that can occur.  

Notice that for any $n$ exactly one of the two sums: 
\begin{equation*}
||h_n||_+^2:=\sum_{k\geq 0}|h_n(k)|^2, \ \ ||h_n||_-^2:=\sum_{k< 0}|h_n(k)|^2
\end{equation*}
must be finite, or otherwise $\eta_n$ will not be in $\ell^2(\Z)$ as at least one of $c_1^\pm(n)$ is not zero. 
Assume then that for some $n_0$ we have $||h_{n_0}||_+<\infty$. Because $||h_{n_0}||=\infty$ we must have $||h_{n_0}||_-=\infty$. Now it follows
from the growth conditions (\ref{growth_cond}) and the formulas for $h_n$ that $\|h_n\|_+<\infty$ for all $n\le n_0$, and similarly $\|h_n\|_-=\infty$ for all $n\ge n_0$. But whenever $\|h_n\|_-=\infty$ we must have $\|h_n\|_+<\infty$ by the square summability of $\eta_n$. Thus if $||h_{n_0}||_+<\infty$ for some $n_0$, we have $\|h_n\|_+<\infty$ for all $n$.

Clearly, the same arguments can be made for $\tilde h_n$. Notice however that:
\begin{equation*}
||h_n||_+||\tilde h_n||_+\geq \sum_{k\geq 0}|h_n(k)||\tilde h_n(k)|=\sum_{k\geq 0}\frac{1}{\beta(n+k)}=\infty
\end{equation*}
by the growth conditions on $\beta(k)$, and similarly for sums with $k<0$. Thus, if $\|h_n\|_+<\infty$ then $\|\tilde h_n\|_+=\infty$ and the other way around. This leads to the following two options:

\underline{Option 1:} For all $n$ we have:
$$\sum_{k>0}|h_n(k)|^2<\infty, \sum_{k<0}|h_n(k)|^2=\infty, \sum_{k>0}|\tilde h_n(k)|^2=\infty, \sum_{k<0}|\tilde h_n(k)|^2<\infty.$$   

\underline{Option 2:} For all $n$ we have: 
$$\sum_{k>0}|h_n(k)|^2=\infty, \sum_{k<0}|h_n(k)|^2<\infty, \sum_{k>0}|\tilde h_n(k)|^2<\infty, \sum_{k<0}|\tilde h_n(k)|^2=\infty.$$ 

Using the estimates (\ref{growth_cond}) for option 1 we see that the following two sums are convergent for every $n$: 
\begin{equation*}
\sum_{k>0}\left |\frac{(k^2+1)^n}{\mu(k)}\right |^2<\infty, \sum_{k<0}\left |\frac{\mu(k)}{(k^2+1)^{n+1}}\right |^2<\infty.
\end{equation*}
This clearly implies the growth estimates in the statement of the lemma. Option 2 is similar, thus the result follows, completing the proof.
\end{proof}

In light of this lemma, we see that the coefficients $\mu(k)$ are functions of rapid decay or rapid growth at different infinities.  

Now that we have control over the coefficients of $D_{\beta,\alpha,1}$ we can finish the main result.  Interestingly, the two options from the previous lemma require different arguments.

\begin{lem} 
Suppose $\mu(k)$ is a function of rapid growth for $k>0$ and a function of rapid decay for $k<0$.  Then the operator $D_{\beta,\alpha,1}$ cannot have compact parametrices.
\end{lem}
\begin{proof}
By Lemma \ref{kernel_property}, the operator $D_n$ has no kernel or cokernel for all $n$.  Therefore we can solve the equation:
\begin{equation*}
(D_nf)(k) = \beta(k+n)\left(f(k) - \frac{\beta(k)}{\beta(k+n)}\frac{\mu(k+1)}{\mu(k)}f(k+1)\right) = g(k)
\end{equation*}
for any $g\in\ell^2(\Z)$.  Using the fact that $\mu(k)$ is a function of rapid decay for $k<0$, and solving recursively, we get the following formulas:
\begin{equation*}
\begin{aligned}
&(D_n^{-1}g)(k) = \\
&\left\{
\begin{aligned}
&-\sum_{j=-\infty}^{k-1} \frac{\beta(k)\cdots\beta(k+n-1)}{\beta(j)\cdots\beta(j+n)}\cdot\frac{\mu(j)}{\mu(k)}g(j) = -\sum_{j=-\infty}^{k-1}\Delta_n^+(k,j)g(j) &&\textrm{ for }n>0 \\
&-\sum_{j=-\infty}^{k-1} \frac{\beta(j-1)\cdots\beta(j+n+1)}{\beta(k-1)\cdots\beta(k+n)}\cdot\frac{\mu(j)}{\mu(k)}g(j) = -\sum_{j=-\infty}^{k-1}\Delta_n^-(k,j)g(j) &&\textrm{ for }n\leq 0,
\end{aligned}
\right.
\end{aligned}
\end{equation*}
where we introduced the notation $\Delta^\pm_n$ for the integral kernels of operators $D_n^{-1}$.
To obtain a contradiction we consider the sub-diagonal matrix coefficients:
\begin{equation*}
\Delta_n^+(k,k-1) = \frac{\beta(k)\cdots\beta(k+n-1)}{\beta(k-1)\beta(k)\cdots\beta(k+n-1)}\cdot\frac{\mu(k-1)}{\mu(k)} = \frac{\mu(k-1)}{\beta(k-1)\mu(k)}
\end{equation*}
which are independent of $n$.  On the other hand we have
\begin{equation*}
|\Delta_n^+(k,k-1)| = |\langle \chi_k, D_n^{-1}\chi_{k-1}\rangle| \le \|D_n^{-1}\|\|\chi_k\|\|\chi_{k-1}\| = \|D_n^{-1}\|,
\end{equation*}
where, as before, $\chi_k(j) = 1$ when $j=k$ and $0$ otherwise.
However, if $D_{\beta,\alpha,1}$ has compact parametrices then $\|D_n^{-1}\|$ goes to $0$ as $n\to\pm\infty$ as $U^{-1}D_{\beta,\alpha,1}$ is a direct sum of operators $D_n$.  This is in clear contradiction with the above calculation. 
\end{proof}

Now we consider option 2 of Lemma \ref{mu_ineq}.
\begin{lem}
Suppose $\mu(k)$ is a function of rapid decay for $k>0$ and a function of rapid growth for $k<0$.  Then the operator $D_{\beta,\alpha,1}$ cannot have compact parametrices.
\end{lem}
\begin{proof}
To obtain a contradiction we study the spectrum of the operator $D_{\beta,\alpha,1}^{max}$.  Since it has compact parametrices, the Fredholm property of $D_{\beta,\alpha,1}^{max}-I$ implies that $\textrm{Ran }((D_n^+)^{max} - \lambda I)$ is closed, which means that the continuous spectrum $\sigma_c((D_n^+)^{max})$ is empty.  

Next we study the eigenvalue equation $(D_nf)(k) = \lambda f(k)$, that is:
\begin{equation}\label{eigen_eq}
\beta(k+n)f(k) - \beta(k)\frac{\mu(k+1)}{\mu(k)}f(k+1) = \lambda f(k).
\end{equation}
Pick a value $\beta(l)$ of the coefficient $\beta$ such that $\beta(j)\ne\beta(l)$ for all $j<l$. This is possible since $\beta(j)\to\infty$ as $j\to-\infty$.
Consider equation (\ref{eigen_eq}) for eigenvalue $\lambda_l = \beta(l)$. Then, substituting $k=l-n$ into (\ref{eigen_eq}), we get that $f(l-n+1)=0$.
Consequently, equation (\ref{eigen_eq}) can be solved recursively, yielding a one-parameter formal solution generated by:
\begin{equation*}
\begin{aligned}
f_l(k) &=\left\{
\begin{aligned}
&\frac{1}{\mu(k)}\left(\prod_{j=k}^{l-n-1}1+ \frac{\beta(j)- \beta(j+n) + \lambda_l}{\beta(j+n)-\lambda_l}\right) &&k\le l-n \\
&0 && k>l-n.
\end{aligned}\right. 
\end{aligned}
\end{equation*}
Notice that by our assumption on $\beta(l)$ the denominators in the above equation are never zero.
The question then is when does $f_l\in\ell^2(\Z)$?
First we estimate from above each factor in the formula for $f_l$ as follows:
\begin{equation*}
\left|1 + \frac{\beta(j)-\beta(j+n)+\lambda_l}{\beta(j+n)-\lambda_l}\right|^2\le \textrm{exp}\left(\left|\frac{\beta(j)-\beta(j+n)+\lambda_l}{\beta(j+n)-\lambda_l}\right|^2\right).
\end{equation*}
This implies that for $k\le l-n$
\begin{equation*}
|f_l(k)|\le \textrm{exp}\left[\sum_{j=k}^{l-n-1}\left|\frac{\beta(j)-\beta(j+n)+\lambda_l}{\beta(j+n)-\lambda_l}\right|^2\right]\frac{1}{|\mu(k)|}.
\end{equation*}

Notice that we have $|\beta(j)-\beta(j+n) +\lambda_l|\le \textrm{const}$ uniformly in $j$ by (\ref{growth_cond}), so, using additionally (\ref{growth_cond}) in the denominator, we obtain:
\begin{equation*}
\sum_{j=k}^{l-n-1} \left|\frac{\beta(j)-\beta(j+n)+\lambda_l}{\beta(j+n)-\lambda_l}\right|^2\le \sum_{j=k}^{l-n-1}\frac{\textrm{const}}{j^2+1}<\textrm{const}.
\end{equation*}
Therefore we have
\begin{equation*}
\sum_{k\in\Z}|f_l(k)|^2 = \sum_{k=-\infty}^{l-n-1}|f_l(k)|^2 \le \sum_{k=-\infty}^{l-n-1}\frac{\textrm{const}}{|\mu(k)|^2}<\infty,
\end{equation*}
and so $f_l\in\ell^2(\Z)$.  Thus $\lambda_l$ is an eigenvalue for $D_n^{max}$ for all $n\in\Z$.  So $\lambda_l\in\sigma(U^{-1}D_{\beta,\alpha,1}^{max})$.  Crucially, since $\lambda_l$ is independent of $n$, we see that $\lambda_l$ has an infinite degeneracy.  As explained in the appendix of \cite{KMRSW1}, if $D_{\beta,\alpha,1}^{max}$ has compact parametrices then its spectrum is either empty, or is the whole plane $\C$, or consist of eigenvalues going to infinity, which is in direct contradiction with the infinite degeneracy. Therefore $D_{\beta,\alpha,1}$ does not have compact parametrices.
\end{proof}
To summarize, we reduced the existence of an operator $D_{\beta,\alpha,w}$ in $H_{\tau_w}$, $D_{\beta,\alpha,w}^{min}\subset D_{\beta,\alpha,w}\subset D_{\beta,\alpha,w}^{max}$ and having compact parametrices to the existence of an operator $D_{\beta,\alpha,1}$ in $\ell^2(\Z)$ with coefficients satisfying (\ref{growth_cond}) and conditions of Lemma \ref{mu_ineq} and having compact parametrices. However, as seen in the last two lemmas, that leads to contradictions and the proof of the theorem is complete.
\end{proof}

\end{document}